\documentclass[reqno]{amsart}
\usepackage{float}
\usepackage{amsmath}
\usepackage{graphicx}
\usepackage{latexsym}
\usepackage{amsfonts}
\usepackage{amssymb}
\theoremstyle{plain}
\newtheorem{theorem}{Theorem}
\newtheorem{corollary}[theorem]{Corollary}

\newtheorem{proposition}[theorem]{Proposition}
\theoremstyle{definition}

\newtheorem*{remark*}{Remark}

\newcommand{\pr}{\mathbf P}
\newcommand{\e}{\mathbf E}

\begin{document}
\title[Maximum of the excursion]{Local tail asymptotics for the joint distribution of length and of maximum of a random walk excursion}

\author[Perfilev]{Elena Perfilev} 
\address{Institut f\"ur Mathematik, Universit\"at Augsburg, 86135 Augsburg, Germany}
\email{Elena.Perfilev@math.uni-augsburg.de}

\author[Wachtel]{Vitali Wachtel} 
\address{Institut f\"ur Mathematik, Universit\"at Augsburg, 86135 Augsburg, Germany}
\email{vitali.wachtel@math.uni-augsburg.de}

\begin{abstract}
This note is devoted to the study of the maximum of the excursion of a random walk with negative drift and light-tailed increments. More precisely, we determine the local
asymptotics of the joint distribution of the length, maximum and the time at which this
maximum is achieved. This result allows one to obtain a local central limit theorems for the length of the excursion conditioned on large values of the maximum.
\end{abstract}
\maketitle
\section{Introduction and statement of main results}
Let $\{S_n; n\geq 1\}$ be a random walk with i.i.d. increments $\lbrace X_k; k\geq 1\rbrace$.
We are interested in the asymptotic properties of some functionals of the excursion
$$
\{S_1,S_2,\ldots,S_\tau\},
$$
where
$$
\tau:=\inf\{n\ge1: S_n\le0\}.
$$
Clearly, $\tau$ is almost surely finite for all random walks satisfying
${\liminf_{n\to\infty}S_n=-\infty}$. This holds, in particular, for all random walks with
non-positive drift.

Denote 
$$
M_n:=\max_{k\leq n}S_k,\quad n\ge1.
$$
The main purpose of this note is to determine the asymptotic behaviour of local probabilities
of the vector $(M_\tau,\theta_\tau,\tau)$, where
$$
\theta_\tau:=\min\{n\ge0:S_n=M_\tau\}.
$$
We shall always assume that $S_n$ is integer-valued, has negative
drift, and satisfies the Cramer condition: there exists $\lambda>0$ such that
\begin{equation}
\label{Cramer.cond}
\varphi(\lambda):=\mathbf{E}e^{\lambda X_1}=1.
\end{equation}
Obviously, \eqref{Cramer.cond} implies that
$$
\e[X_1]=:-a<0.
$$

It is well-known, see Iglehart~\cite{Igl72}, that if, additionally, $\mathbf{E}[X_1e^{\lambda X_1}]<\infty$ then there
exists $c_0\in(0,1)$ such that, as $x\to\infty$,
\begin{equation}
\label{Cramer.local}
\mathbf{P}(M_\tau=x)\sim c_0 e^{-\lambda x},
\end{equation}
provided that the distribution of $X_1$ is aperiodic.

Moreover, according to Theorem II in Doney \cite{Doney89}, one has, for all aperiodic walks
satisfying \eqref{Cramer.cond},
\begin{equation}
\label{Doney.local}
\mathbf{P}(\tau=n)\sim c_1 n^{-3/2}\left(\mathbf{E}e^{\mu X_1}\right)^{n},
\end{equation}
where $\mu>0$ is uniquely determined by $\mathbf{E}[X_1e^{\mu X_1}]=0$.

These marginal asymptotics do not allow one to guess the right asymptotic behaviour of the
joint distribution. The reason is a very strong dependence between large values of $\tau$
and $M_\tau$. This can be illustrated by the optimal strategy for the occurence of the event
$\{M_\tau=x\}$.  First, the random walk goes up linearly with the rate
$$
\widehat{a}:=\mathbf{E}[X_1e^{\lambda X_1}]/\mathbf{E}[e^{\lambda X_1}].
$$
After reaching the level $x$, the random walk goes down with the standard rate $a$. As a result,
the stopping time $\tau$ is of the order $x/\widehat{a}+x/a$ on $\{M_{\tau}>x\}.$

According to Theorem 5.2 in Asmussen~\cite{Asmussen82}, conditioned on the large value of the maximum $M_\tau$, $\theta_\tau$ satisfies a central limit theorem. More precisely, 
\begin{equation}
\label{theta-clt}
\mathbf{P}\left(\frac{\theta_\tau-x/\widehat{a}}{\widehat{c}\sqrt{x}}<u\Big|M_\tau>x\right)
\to\Phi(u),\quad x\to\infty,
\end{equation}
where
$$
\widehat{c}:=\sqrt{\frac{\widehat{\sigma}^2}{\widehat{a}^3}}
\quad\text{and}\quad
\widehat{\sigma}^2:=\mathbf{E}[X_1^2e^{\lambda X_1}]-\widehat{a}^2.
$$
Clearly, this result improves the description of the first part of the optimal strategy leading to
$\{M_\tau>x\}$.

Our main result describes the asymptotic behaviour of the mass function of the vector $(M_\tau,\theta_\tau,\tau)$.
\begin{theorem}
\label{thm:excursion}
Assume that \eqref{Cramer.cond} holds and, furthermore, 
$$
\widehat{\sigma}^2=\mathbf{E}[X_1^2e^{\lambda X_1}]-\widehat{a}^2\in(0,\infty).
$$
 Assume also that the distribution of $X_1$ is aperiodic. Then there exists $Q>0$ such that,
 as $x\to\infty$,
 \begin{align*}
&e^{\lambda x}\pr(M_\tau=x,\theta_\tau=k,\tau=n+1)=\\
&\hspace{1cm}\frac{Q}{2\pi\sqrt{k(n-k)}}\exp\left\lbrace-\frac{(x-k\widehat{a})^2}{2\widehat{\sigma}^2k}-\frac{(x-a(n-k))^2}{2\sigma^2(n-k)}\right\rbrace+o\left(\frac{1}{\sqrt{k(n-k)}}\right)
\end{align*}
uniformly in $k,n$ such that $k<n$. The exact form of the constant $Q$ is given in \eqref{T8.6}.
\end{theorem}
The proof strategy in Theorem~\ref{thm:excursion}  enables one to prove local limit theorems for $\tau$ and $\theta_\tau$ conditioned on
$M_\tau=x$.
\begin{corollary}
\label{cor:tau-theta}
Under the conditions of Theorem~\ref{thm:excursion},
as $x\to\infty$,
\begin{equation}
\label{theta-local}
\sqrt{x}\mathbf{P}(\theta_\tau=k|M_\tau=x)=\frac{\widehat{a}^{3/2}}{\sqrt{2\pi\widehat\sigma^2}}
\exp\left\{-\frac{\widehat{a}^3(k-x/\widehat{a})^2}{2x\widehat\sigma^2}\right\}+o(1)
\end{equation}
and 
\begin{equation}
\label{theta-tau-local}
\sqrt{x}\mathbf{P}(\tau-\theta_\tau=k|M_\tau=x)=\frac{a^{3/2}}{\sqrt{2\pi\sigma^2}}
\exp\left\{-\frac{a^3(k-x/a)^2}{2x\sigma^2}\right\}+o(1)
\end{equation}
uniformly in $k$. Furthermore, conditionally on $M_\tau=x$, the random variables $\theta_\tau$ and $\tau-\theta_\tau$
are independent. Consequently,
\begin{equation}
\label{tau-local}
\sqrt{x}\mathbf{P}(\tau=n|M_\tau=x)=\frac{1}{\sqrt{2\pi\Sigma^2}}
\exp\left\{-\frac{(n-Ax)^2}{2x\Sigma^2}\right\}+o(1),
\end{equation}
where 
$$
A:=\frac{1}{a}+\frac{1}{\widehat{a}}
\quad\text{and}\quad
\Sigma^2:=\frac{\sigma^2}{a^3}+
\frac{\widehat\sigma^2}{\widehat a^3}.
$$
\end{corollary}

Since the tail of $M_\tau$ is exponentially decreasing, one infers from \eqref{theta-clt} that
\begin{equation}
\label{loc.cond}
\mathbf{P}\left(\frac{\theta_\tau-x/\widehat{a}}{\widehat{c}\sqrt{x}}<u\Big|M_\tau=x\right)
\to\Phi(u),\quad x\to\infty.
\end{equation}
Thus, \eqref{theta-local} is a local version of this central limit theorem for $\theta_\tau$.

Our approach to the excursions of random walks satisfying the Cramer condition \eqref{Cramer.cond}
is based on the standard change of measure:
$$
\widehat{\mathbf{P}}(X_1\in dx)=e^{\lambda x}\mathbf{P}(X_1\in dx).
$$
Under this new measure  the drift of $S_n$ becomes positive:
$\widehat{\mathbf{E}}X_1=\widehat{a}>0.$  Furthermore, the variance of $X_k$ under the new measure
becomes equal to $\widehat{\sigma}^2$. For that reason we  study in the next section local probability
asymptotics for conditioned random walks with positive drift.  These results are later used in the proof of
Theorem~\ref{thm:excursion}, which is given in Section 3.

\section{Local limit theorems for functionals of a random walk with positive drift}
In this section we shall always assume that
$\mathbf{E}X_1=a>0$ and $\mathbf{E}(X_1-a)^2=\sigma^2\in(0,\infty).$
\subsection{Local limit theorem for a walk conditioned to stay positive}
Define the stopping times
$$
\tau_z=\inf\lbrace n\geq 1:z+S_n\leq 0\rbrace,\hspace{0,5cm} z\geq 0.
$$
In the case of random walks with positive drift one has $\mathbf{P}(\tau_z=\infty)>0$
and, consequently, the asymptotic behaviour of $S_n$ conditioned to stay positive should
be the same as for the unconditioned walk. In the  case $z=0$ Iglehart~\cite{DonIgl} has
shown that for conditioned random walk the standard form of the functional central limit
theorem is still valid, see Proposition 2.1 in \cite{DonIgl}. Our first result shows that CLT for marginals holds for all starting points $z$.
\begin{proposition}\label{genIgl}
Assume that $a>0$ and that $\sigma^2\in(0,\infty)$. Then, for every $z\ge0$,
\begin{equation}\label{conddist}
\pr\left(\frac{S_n-na}{\sigma\sqrt{n}}\leq x\Big\vert\tau_z>n\right)
\rightarrow\Phi(x),\hspace{0,5cm}x\in\mathbb{R},
\end{equation}
where $\Phi(x)$ is the the standard normal distribution function.
\end{proposition}
\begin{proof}
As in the proof of Proposition~2.1 from \cite{DonIgl}, we shall use the decomposition
\begin{align}
\label{decomp}
\nonumber
&\pr\left(\frac{S_n-na}{\sigma\sqrt{n}}\leq x,\tau_z>n\right)\\
&\hspace{1cm}=\pr\left(\frac{S_n-na}{\sigma\sqrt{n}}\leq x\right)
-\sum_{k=1}^n\pr\left(\frac{S_n-na}{\sigma\sqrt{n}}\leq x,\tau_z=k\right).
\end{align}
By the central limit theorem,
\begin{equation}
\label{normdens}
\pr\left(\frac{S_n-na}{\sigma\sqrt{n}}\leq x\right)\rightarrow\Phi(x).
\end{equation}
By the Markov property, for every fixed $k$,
\begin{align*}
\pr\left(\frac{S_n-na}{\sigma\sqrt{n}}\leq x, \tau_z=k \right)
=\int_{-\infty}^{-z}\pr(S_k\in dy,\tau_z=k)\pr\left(\frac{y+S_{n-k}-na}{\sigma\sqrt{n}}\leq x\right).
\end{align*}
Then, using \eqref{normdens} and the dominated convergence, we obtain
$$
\pr\left(\frac{S_n-na}{\sigma\sqrt{n}}\leq x,\tau_z=k\right)\rightarrow\Phi(x)\pr(\tau_z=k).
$$
Consequently, for every fixed $N>1$,
\begin{align}\label{1.part}
\nonumber
&\pr\left(\frac{S_n-na}{\sigma\sqrt{n}}\leq x\right)
-\sum_{k=1}^N\pr\left(\frac{S_n-na}{\sigma\sqrt{n}}\leq x,\tau_z=k\right)\\
&\hspace{1cm}\rightarrow\Phi(x)\left(1-\sum_{k=1}^N\pr(\tau_z=k)\right)=\Phi(x)\pr(\tau_z>N).
\end{align}
For the tail of the sum one has, uniformly in $n>N$,
\begin{align}
\label{2.part}
\nonumber
0<\sum_{k=N}^n\pr\left(\frac{S_n-na}{\sigma\sqrt{n}}\leq x,\tau_z=k\right)
&\leq\sum_{k=N}^n\pr(\tau_z=k)\\
&=\pr(N\leq\tau_z\leq n)\leq \pr(N\leq\tau_z<\infty).
\end{align}
Combining \eqref{1.part} and \eqref{2.part} and letting $N\to\infty$, we obtain
\begin{align*}
\pr\left(\frac{S_n-na}{\sigma\sqrt{n}}\leq x,\tau_z>n\right)\rightarrow\Phi(z)\pr(\tau_z=\infty).
\end{align*}
The positivity of the drift implies that $\pr(\tau_z=\infty)$ is positive. Therefore, the previous convergence
is equivalent to the statement of the proposition.
\end{proof}
We now turn to the corresponding local limit theorem.
\begin{theorem}
\label{T.1}
Assume that $a>0$ and that $\sigma^2\in(0,\infty)$. Assume also, that the distribution of $X_1$ is aperiodic.
Then, uniformly in $x>-z$, 
\begin{align}
\label{main}
\pr(S_n=x,\tau_z>n)=\frac{\mathbf{P}(\tau_z=\infty)}{\sqrt{2\pi\sigma^2n}}e^{-(x-na)^2/2\sigma^2n}
+o\left(\frac{1}{\sqrt{n}}\right).
\end{align}
\end{theorem}
\begin{proof}
Set $m=\left[\frac{n}{2}\right]$. By the Markov property,
\begin{equation}\label{split.sum}
\begin{split}
\pr(S_n&=x, \tau_z>n)=\\
&=\sum_{y=1-z}^\infty\pr\left(S_m=y,\tau_z>m\right)\pr\left(S_{n-m}=x-y,\min_{k\leq n-m}S_k\geq -y-z\right).
\end{split}
\end{equation}
We split the sum in \eqref{split.sum} into two parts.
First,
\begin{align}\label{sum.1}
\nonumber
\sum_{y=1-z}^{am/2}\pr(S_m=y,\tau_z>m)&\pr(S_{n-m}=x-y,\min_{k\leq x-y}S_k\geq-y-z)\\
&\leq \sup_u\pr(S_{n-m}=u)\pr\left(S_m\leq\frac{am}{2},\tau_{z}>m\right).
\end{align}
Using the Chebyshev inequality we infer that
\begin{equation}\label{max.cheb}
\pr\left(S_m\leq\frac{am}{2},\tau_z>m\right)
\le \pr(S_m\leq\frac{am}{2})\rightarrow 0,\hspace{0,5cm}\text{as }m\rightarrow\infty.
\end{equation}
Furthermore, by the local limit theorem for unconditioned random walks, 
\begin{equation}
\label{concentration}
\sup_z\pr(S_{n-m}=z)\leq\frac{c}{\sqrt{n-m}}.
\end{equation}
Plugging this estimate into \eqref{sum.1} and using \eqref{max.cheb}, we get
\begin{equation}\label{sum.1*}
\sum_{y=1-z}^{am/2}\pr(S_m=y,\tau_z>m)\pr(S_{n-m}=x-y,\min_{k\leq n-m}S_k\geq-y-z)=o\left(\frac{1}{\sqrt{n}}\right).
\end{equation}
Second, for $y>am/2$ we shall use the representation
\begin{align*}
&\pr\left(S_{n-m}=x-y,\min_{k\leq n-m} S_k\geq -y-z\right)\\
&\hspace{1cm}=\pr\left(S_{n-m}=x-y\right)-\pr\left(S_{n-m}=x-y,\min_{k\leq x-y} S_k<-y-z\right).
\end{align*}
Therefore,
\begin{equation*}
\pr\left(S_{n-m}=x-y,\min_{k\leq n-m} S_k\geq-y-z\right)\leq\pr(S_{n-m}=x-y),
\end{equation*}
and
\begin{equation*}
\begin{split}
\pr\left(S_{n-m}=x-y,\min_{k\leq n-m} S_k\geq-y-z\right)&\\
\geq \pr\left(S_{n-m}=x-y\right)
-\pr&\left(\min_{k\leq n-m} S_k<-\frac{am}{2}\right).
\end{split}
\end{equation*}
Applying the classical Kolmogorov inequality, we get
\begin{align}\label{max.doob}
\nonumber
\pr\left(\min_{k\leq n-m} S_k<-\frac{am}{2}\right)
&\leq\pr\left(\min_{k\leq n-m}(S_k-\e[S_k])<-\frac{am}{2}\right)\\
&\leq\frac{\mathbf{Var}[S_{n-m}]}{a^2m^2/4}\leq\frac{c}{n}.
\end{align}
Consequently,
\begin{equation}\label{sum.2}
\begin{split}
\sum_{y\geq\frac{am}{2}}\pr(S_m=y,\tau_z>m)\pr\left(S_{n-m}=x-y,\min_{k\leq n-m}S_k>-y-z\right)=\\
=\sum_{y\geq\frac{am}{2}}\pr(S_m=y,\tau_z>m)\pr(S_{n-m}=x-y)+O\left(\frac{1}{n}\right).
\end{split}
\end{equation}
Combining \eqref{sum.1*} and \eqref{sum.2} we get
\begin{equation*}
\pr(S_n=x,\tau_z>n)=\sum_{y=1}^\infty\pr(S_m=y,\tau_z>m)\pr(S_{n-m}=x-y)+o\left(\frac{1}{\sqrt{n}}\right).
\end{equation*}
By the local limit theorem for $\{S_n\}$,
\begin{equation*}
\pr(S_{n-m}=x-y)=\frac{1}{\sqrt{2\pi(n-m)\sigma^2}}e^{-(x-y-a(n-m))^2/2\sigma^2(n-m)}+o\left(\frac{1}{\sqrt{n}}\right)
\end{equation*}
uniformly in $x-y$.

Therefore,
\begin{equation}\label{sum.3}
\begin{split}
\pr(S_n&=x,\tau_z>n)\\
&=\frac{\pr(\tau_z>m)}{\sqrt{2\pi(n-m)\sigma^2}}\e\left[e^{-(x-S_m-an+am)^2/2\sigma^2(n-m)}\vert \tau_z>m\right]+o\left(\frac{1}{\sqrt{n}}\right).
\end{split}
\end{equation}
Using now Proposition~\ref{genIgl}, we get
\begin{align}
\label{propIgl}
\nonumber
&\e\left[e^{-(x-an-(S_m-am))^2/2(n-m)\sigma^2}\vert\tau_z>m\right]\\
\nonumber
&\hspace{2cm}=\int_{-\infty}^\infty
e^{-\left(\frac{x-an}{\sqrt{(n-m)\sigma^2}}-u\right)^2/2}\frac{1}{\sqrt{2\pi}}
e^{-u^2/2}+o(1)\\
&\hspace{2cm}=\frac{1}{\sqrt{2}}e^{-(x-an)^2/2n\sigma^2}+o(1).
\end{align}
Plugging \eqref{propIgl} into \eqref{sum.3}, we obtain
\begin{equation*}
\pr(S_n=x,\tau_z>n)=\frac{\pr(\tau_z>m)}{\sqrt{2\pi\sigma^2n}}e^{-(x-an)^2/2n\sigma^2}+o\left(\frac{1}{\sqrt{n}}\right).
\end{equation*}
Recalling that $\pr(\tau_z>m)\rightarrow\pr(\tau_z=\infty)$, we finally get the relation \eqref{main}. 
\end{proof}
\subsection{Local asymptotics for $(M_n,S_n)$}
Define 
$$
\tau_+:=\min\{n\geq1:S_n>0\}.
$$
\begin{theorem}\label{T.1.1}
Under the assumptions of Theorem \ref{T.1}, uniformly in $x\ge0$ and $r\ge0$,
\begin{equation*}
\pr(M_n=x,S_n=x-r)=
\frac{\mathbf{P}(\tau_+=\infty)V(r)}{\sqrt{2\pi\sigma^2n}}e^{-(x-na)^2/2\sigma^2n}+o\left(\frac{1}{\sqrt{n}}\right),
\end{equation*}
where 
$$
V(r):=\sum_{j=0}^\infty\pr(S_j=-r,\tau_+>j).
$$
\end{theorem}
\begin{proof}
Let $\theta_n$ be the first time the random walk achieves its maximum. That is,
$$
\theta_n:=\min\lbrace k\geq 1: S_k=M_n\rbrace.
$$
By the Markov property,
\begin{align}
\label{max.1}
\nonumber
&\pr(M_n=x, S_n=x-r)\\
\nonumber
&\hspace{0.5cm}=\sum_{k=0}^n\pr(M_n=x,S_n=x-r,\theta_n=k)\\
\nonumber
&\hspace{0.5cm}=\sum_{k=0}^n \pr(S_k=x, S_j<S_k\ \text{for all }j<k)
\pr(S_{n-k}=-r, S_j\leq 0\ \text{for all }j\leq n-k)\\
&\hspace{0.5cm}=\sum_{k=0}^n\pr(S_k=x, S_k-S_j>0\ \text{for all }j<k)
\pr(S_{n-k}=-r,\tau_+>n-k).
\end{align}
It follows from the duality lemma for random walks that
\begin{equation}
\label{duallem}
\pr(S_k=x, S_k-S_j>0\ \text{for all }j<k)=\pr(S_k=x, \tau>k).
\end{equation}
Combining \eqref{max.1} and \eqref{duallem}, we obtain
\begin{equation}\label{max.2}
\pr(M_n=x, S_n=x-r)=\sum_{k=0}^n\pr(S_k=x, \tau>k)\pr(S_{n-k}=-r,\tau_+>n-k).
\end{equation}
Fix some $\varepsilon\in(0,1/2)$.
The assumption $a>0$ implies that $\mathbf{E}\tau_+$ is finite. Therefore, uniformly in $k\leq(1-\varepsilon)n$ and $r\ge0$,
\begin{align*}
\pr(S_{n-k}=-r, \tau_+>n-k)\leq\pr(\tau_+>\varepsilon n)=o\left(\frac{1}{n}\right).
\end{align*}
Consequently, using the transience of $\lbrace S_n\rbrace$,
\begin{align}
\label{max.3}
\nonumber
&\sum_{k=0}^{(1-\varepsilon)n}\pr(S_k=x,\tau>k)\pr(S_{n-k}=-r,\tau_+>n-k)\\
&\hspace{3cm}\le \pr(\tau_+>\varepsilon n)\sum_{k=0}^\infty\pr(S_k=x)
=o\left(\frac{1}{n}\right).
\end{align}
If $k\ge(1-\varepsilon)n$ then, by \eqref{concentration},
\begin{equation}
\label{max3.1}
\pr(S_k=x,\tau>k)\leq\pr(S_k=x)\leq\frac{c}{\sqrt{(1-\varepsilon)n}}\leq\frac{2c}{\sqrt{n}}.
\end{equation}
This bound implies that, for every fixed $N$,
\begin{align}
\label{max.4}
\nonumber
&\sum_{(1-\varepsilon)n}^{n-N}\pr(S_k=x,\tau>k)\pr(S_{n-k}=-r,\tau_+>n-k)\\
&\hspace{1cm}\leq \frac{c}{\sqrt{n}}\sum_{j=N}^{\varepsilon n}\pr(S_j=-r,\tau_+>j)
\leq\frac{c}{\sqrt{n}}\sum_{j=N}^\infty\pr(\tau_+>j).
\end{align}
We also note that the finiteness of $\e\tau_+$ yields
\begin{equation}
\label{eps_N}
\lim_{N\to\infty}\sum_{j=N}^\infty\pr(\tau_+>j)=0.
\end{equation}
Using Theorem \ref{T.1} with $z=0$, we obtain
\begin{align}
\label{max.5}\nonumber
&\sum_{k=n-N}^n\pr(S_k=x,\tau>k)\pr(S_{n-k}=-r,\tau_+>n-k)\\
&\hspace{1cm}=\frac{\mathbf{P}(\tau=\infty)}{\sqrt{2\pi\sigma^2n}}e^{-(x-na)^2/2\sigma^2n}
\sum_{j=0}^N\pr(S_j=-r,\tau_+>j)+o\left(\frac{1}{\sqrt{n}}\right).
\end{align}
Plugging \eqref{max.3}, \eqref{max.4} and \eqref{max.5} into \eqref{max.2}, we obtain
\begin{equation}\label{max.6}
\begin{split}
&\left|\pr(M_n=x,S_n=x-r)-\frac{\pr(\tau=\infty)}{\sqrt{2\pi\sigma^2n}}e^{(x-na)^2/2\sigma n}
\sum_{j=0}^N\pr(S_j=-r,\tau_+>j)\right|\\
&\hspace{3cm}\le\sum_{j=N}^\infty\pr(\tau_+>j)+o\left(\frac{1}{\sqrt{n}}\right).
\end{split}
\end{equation}
Letting now $N\rightarrow\infty$ and taking into account \eqref{eps_N}, we arrive at \eqref{max.6}. 
\end{proof}
\begin{corollary}
\label{prop:Mn}
Under the conditions of Theorem~\ref{T.1},
$$
\mathbf{P}(M_n=x)=\frac{1}{\sqrt{2\pi\sigma^2n}}e^{-(x-an)^2/2n}
+o\left(\frac{1}{\sqrt{n}}\right)
$$
uniformly in $x\ge0$.
\end{corollary}
\begin{proof}
Similar to the proof of Theorem~\ref{T.1.1},
\begin{align}
\label{max.7}
\nonumber
\pr(M_n=x)&=\sum_{k=0}^n\pr(M_n=x,\theta_n=k)\\
\nonumber
&=\sum_{k=0}^n\pr(S_k=x, S_k-S_j>0\text{ for all }j<k)\pr(\tau_+>n-k)\\
&=\sum_{k=0}^n\pr(S_k=x,\tau>k)\pr(\tau_+>n-k).
\end{align}
The bound
\begin{equation*}
\pr(\tau_+>\varepsilon n)=o\left(\frac{1}{n}\right),
\end{equation*}
implies that
\begin{align}
\label{max.8}
\nonumber
&\sum_{k=0}^{(1-\varepsilon)n}\pr(S_k=x,\tau>k)\pr(\tau_+>n-k)\\
&\hspace{2cm}\le\pr(\tau_+>\varepsilon n)\sum_{k=0}^\infty\pr(S_k=x)=o\left(\frac{1}{n}\right),
\end{align}
in the last step we have also used the fact that $\lbrace S_k\rbrace$ is transient.
Using the same arguments as in the derivation of \eqref{max3.1}, 
\begin{equation}\label{max.9}
\sum_{(1-\varepsilon)}^{n-N}\pr(S_k=x, \tau>k)\pr(\tau_+>n-k)
\leq\frac{c}{\sqrt{n}}\sum_{j=N}^\infty\pr(\tau_+>j)=\frac{\varepsilon_N}{\sqrt{n}}.
\end{equation}
Finally, using \eqref{main}, we obtain
\begin{align}
\label{max.10}
\nonumber
&\sum_{k=n-N}^n\pr(S_k=x,\tau>k)\pr(\tau_+>n-k)\\
&\hspace{2cm}=\frac{\mathbf{P}(\tau=\infty)}{\sqrt{2\pi\sigma^2n}}e^{-(x-na)^2/2\sigma^2n}\sum_{j=0}^N\pr(\tau_+>j).
\end{align}
Combining \eqref{max.7},\eqref{max.8}, \eqref{max.9} and \eqref{max.10}
and letting $N\rightarrow\infty$, we conclude that
\begin{equation*}
\pr(M_n=x)=\frac{\mathbf{P}(\tau=\infty)\e\tau_+}{\sqrt{2\pi\sigma^2n}}e^{-(x-na)^2/2\sigma^2n}
+\left(\frac{1}{\sqrt{n}}\right).
\end{equation*}
Noting that the duality of stopping times $\tau$ and $\tau_+$ implies that
$$
\left(1-\mathbf{E}z^{\tau_+}\right)\left(1-\mathbf{E}z^{\tau}\right)=1-z.
$$
Dividing both parts by $1-z$ and letting $z\to1$, we get
\begin{equation}
\label{pE}
\mathbf{P}(\tau=\infty)\e\tau_+=1.
\end{equation}
This equality completes the proof.
\end{proof}
\begin{corollary}\label{cor:Mn-Sn}
If $x-na=O(\sqrt{n})$, then for every fixed $r\ge0$
\begin{equation}\label{Cor.1}
\lim_{n\to\infty}\pr(S_n=x-r\vert M_n=x)=\mathbf{P}(\tau_z=\infty)V(r)=\frac{V(r)}{\e\tau_+}.
\end{equation}
\end{corollary}
\begin{proof}
Combining Theorem~\ref{T.1.1} and Corollary~\ref{prop:Mn} and using \eqref{pE}, we obtain the desired relation.
\end{proof}
Consider the sequence 
$$
R_n:=M_n-S_n, \quad n\ge0.
$$
It is well-known and easy to see that this sequence can be defined by 
$$
R_{n+1}=(R_n-X_{n+1})^+.
$$
Furthermore, for every $n$ the distribution of $R_n$ is equal to that of $\max_{k\le n}(-S_k)$.
The assumption $a>0$ implies now that, as $n\to\infty$,
$$
\mathbf{P}(M_n-S_n=r)=\mathbf{P}(R_n=r)\to\mathbf{P}\left(\max_{k\ge1}(-S_k)=r\right),\quad r\ge0.
$$
By the ladder heights representation for $\max_{k\ge1}(-S_k)$ and by the duality lemma,
\begin{align*}
\mathbf{P}\left(\max_{k\ge1}(-S_k)=r\right)
&=\sum_{n=0}^\infty\pr(S_n=-r,\, n\text{ is a descending ladder epoch })\pr(\tau=\infty)\\
&=\sum_{j=0}^\infty\pr(S_j=-r,\tau_+>j)\pr(\tau=\infty)
=\frac{V(r)}{\e\tau_+}.
\end{align*}
As a result, as $n\to\infty$,
$$
\pr(M_n-S_n=r)\to\frac{V(r)}{\e\tau_+}
$$
for every $r\ge0$. Obviously, this classical  relation is a consequence of our Corollary~\ref{cor:Mn-Sn}. 

\begin{theorem}\label{snmax}
Under the conditions of Theorem~\ref{T.1}, for fixed non-negative numbers $y,z$,
\begin{align*}
&\pr(S_n=x, M_{n-1}< x+y,\tau_z>n)\\
&\hspace{1cm}=\frac{\pr(\tau_y=\infty)\pr(\tau_z=\infty)}{\sqrt{2\pi\sigma^2n}}e^{-(x-na)^2/2\sigma^2n}
+o\left(\frac{1}{\sqrt{n+x}}\right)
\end{align*}
uniformly in $x$.
\end{theorem}
\begin{proof}
Fix $\varepsilon>0$. If $\vert x-na\vert>n\varepsilon$ then by the Chebyshev inequality
\begin{align}
\pr(S_n=x, M_{n-1}<x+y,\tau_z>n)&\leq\pr\left(\vert S_n-an\vert\geq\vert x-an\vert\right)=o\left(\frac{1}{n+x}\right).
\end{align}
Thus, it remains to consider the case $\vert x-na\vert\leq\varepsilon n$. Set again $m=\left[n/2\right].$
By the Markov property,
\begin{align}\label{split.sum2}
\nonumber
&\pr(S_n=x, M_{n-1}<x+y, \tau_z>n)\\
\nonumber
&=\sum_{u=1-z}^{x+y-1}\pr(S_m=u, M_m< x+y,\tau_z>m)\\
\nonumber
&\hspace{2cm}\times\pr(S_{n-m}=x-u,M_{n-m}<x+y-u,\tau_{z+u}>n-m)\\
&=\Sigma_1+\Sigma_2,
\end{align}
where $\Sigma_1$ is the sum over $u\in\left(-z,[an/4]\right]$, and $\Sigma_2$ over $u\in\left([an/4]+1, x+y-1\right)$.

Using the Chebyshev inequality once again, we obtain
$$\Sigma_1\leq\pr\left(S_m\leq\left[\frac{an}{4}\right]\right)=o\left(\frac{1}{n}\right).$$
Therefore,
$$\pr(S_n=x, M_{n-1}<x+y,\tau_z>n)=\Sigma_2+o\left(\frac{1}{n}\right).$$
By the Kolmogorov inequality, for $x\ge (a-\varepsilon)n$ and $y\ge0$,
\begin{align*}
\pr(M_m\geq x+y)=o\left(\frac{1}{n}\right).
\end{align*}
Consequently,
\begin{align}\label{first.term}
\nonumber
\pr(S_m&=u, M_m< x+y, \tau_z>m)\\
\nonumber
&=\pr(S_m=u,\tau_z>m)-\pr(S_m=u,M_m\ge x+y,\tau>m)\\
&=\pr(S_m=u,\tau_z>m)+o\left(\frac{1}{n}\right).
\end{align}
Using once again the Kolmogorov inequality, we get
$$
\mathbf{P}(\tau_{z+u}<n-m)
=\pr\left(\min_{k\le n-m} S_k<-z-u\right)
=o\left(\frac{1}{n}\right)
$$
uniformly in $u\ge an/4$.

Therefore,
\begin{align}\label{rev.time}
\nonumber
&\pr(S_{n-m}=x-u,M_{n-m}<x+y-u,\tau_{z+u}>n-m)=\\
&\hspace{1cm}\pr(S_{n-m}=x-u,M_{n-m}<x+y-u)+o\left(\frac{1}{n}\right).
\end{align}
Furthermore, by the duality lemma,
\begin{align*}
\pr(S_{n-m}=x-u, M_{n-m}<x+y-u)=\pr(S_{n-m}=x-u,\tau_y>n-m).
\end{align*}
Plugging these equalities into $\Sigma_2$, we conclude that
\begin{align*}
\Sigma_2=\sum_{u=[an/4]+1}^{x+y-1}\pr(S_m=u,\tau_z>m)\pr(S_{n-m}=x-u,\tau_y>n-m)+o\left(\frac{1}{n}\right).
\end{align*}
Applying now Theorem \ref{T.1}, we finally get
\begin{align*}
\Sigma_2&=\frac{\pr(\tau_z=\infty)\pr(\tau_y=\infty)}{2\pi\sigma^2\frac{n}{2}}\sum_{u=[an/4]}^{x+y-1}e^{-(u-am)^2/2\sigma^2m}e^{-(x-u-a(n-m))^2/2\sigma^2(n-m)}\\
&\hspace{2cm}+o\left(\frac{1}{\sqrt{n}}\right)\\
&=\frac{\pr(\tau_z=\infty)\pr(\tau_y=\infty)}{\sqrt{2\pi\sigma^2n}}e^{-(x-an)^2/2\sigma^2n}+o\left(\frac{1}{\sqrt{n}}\right).
\end{align*}
This completes the proof.
\end{proof}
\section{Proofs of main results.}
\subsection{Proof of Theorem~\ref{thm:excursion}}
Clearly,
\begin{align}\label{T8.0}
\nonumber
&\pr(M_\tau=x, \theta_\tau=k,\tau=n+1)\\
\nonumber
&\hspace{1cm}=\pr(M_n=x,\theta_n=k,\tau=n+1)\\
&\hspace{1cm}=\sum_{y=1}^x\pr(M_n=x, \theta_n=k, S_n=y,\tau>n)\pr(X_{n+1}\leq-y).
\end{align}
Furthermore, for every $y\in\lbrace1,2\cdots,x\rbrace$ we have
\begin{align}\label{T8.1}
\nonumber
&\pr(M_n=x, \theta_n=k, S_n=y,\tau>n)\\
&\hspace{1cm}=\pr(S_k=x,\theta_k=k,\tau>k)\pr(S_{n-k}=y-x,M_{n-k}\leq 0,\min_{j\leq n-k} S_j>-x).
\end{align}
Consider a new measure $\widehat{\pr}$ given by
$$\widehat{\pr}(X_k\in du)=\frac{e^{\lambda u}}{\varphi({\lambda})}\pr(X_k\in du),\quad k\ge1.$$
Then one has
\begin{align}
\label{T8.1a}
\nonumber
\pr(S_k=x, \theta_k=k,\tau>k)&=e^{-\lambda x}\widehat{\pr}(S_k=x,\theta_k=k,\tau>k)\\
&=e^{-\lambda x}\widehat{\pr}(S_k=x,M_{k-1}<x,\tau>k).
\end{align}
Applying Theorem~\ref{snmax} with $z=y=0$, we obtain, uniformly in $x$,
\begin{align}\label{T8.2}
e^{\lambda x}\pr(S_k=x,\theta_k=k,\tau>k)
=\frac{\widehat{\pr}^2(\tau=\infty)}{\sqrt{2\pi\widehat{\sigma}^2k}}e^{-(x-k\widehat{a})^2/2\widehat{\sigma^2}k}
+o\left(\frac{1}{\sqrt{x+k}}\right).
\end{align}
For every $k\geq0$ set $\overline{S}_k=-S_k$. Then
\begin{align}\label{T8.2a}
\nonumber
&\pr(S_{n-k}=y-x,M_{n-k}\leq0,\min_{j\leq n-k}S_j>-x)\\
\nonumber
&\hspace{1cm}=\pr(\overline{S}_{n-k}=x-y,\min_{j\leq n-k}\overline{S}_j\geq 0,\max_{j\leq n-k}\overline{S}_j<x)\\
&\hspace{1cm}=\pr(\overline{S}_{n-k}=x-y,\overline{\tau}_1>n-k,\overline{M}_{n-k}<x).
\end{align}
Applying Theorem~\ref{snmax} to the random walk $\lbrace\overline{S}_n\rbrace$, we get, for every fixed $y$,
\begin{align}\label{T8.3}
\nonumber
&\pr(\bar{S}_{n-k}=x-y,\bar{\tau}_1>n-k,\bar{M}_{n-k}<x)\\
&\hspace{0.5cm}=\frac{\pr(\overline\tau_1=\infty)\pr(\overline\tau_y=\infty)}{\sqrt{2\pi\sigma^2(n-k)}}
e^{-(x-a(n-k))^2/2\sigma^2(n-k)}+o\left(\frac{1}{\sqrt{x+(n-k)}}\right).
\end{align}
Furthermore, by \eqref{concentration},
\begin{align*}
\widehat{\pr}(S_k=x,\theta_n=k,\tau>k)\leq\widehat{\pr}(S_k=x)\leq\frac{c}{\sqrt{k}}
\end{align*}
and
\begin{align*}
\pr(\overline{S}_{n-k}=x-y,\overline{\tau}_1>n-k,\overline{M}_{n-k}<x)
\leq\pr(\overline{S}_{n-k}=x-y)\leq\frac{c}{\sqrt{n-k}}.
\end{align*}
Combining these estimates with \eqref{T8.1} and \eqref{T8.2a}, we obtain
\begin{align}\label{T8.4}
\nonumber
e^{\lambda x}\sum_{y=N+1}^x\pr(M_n&=x,\theta_n=k,\tau>n)\pr(X_{n+1}\leq-y)\\
&\leq \frac{c}{\sqrt{k(n-k)}}\sum_{y=N+1}^\infty\pr(X_{n+1}\leq-y).
\end{align}
Combining \eqref{T8.2} and \eqref{T8.3}, we conclude that
\begin{align}\label{T8.5}
\nonumber
&e^{\lambda x}\sum_{y=1}^N\pr(M_n=x, S_n=y,\theta_n=k,\tau>n)\pr(X_{n+1}\leq-y)\\
\nonumber
&\hspace{1cm}
=\frac{\widehat{\pr}^2(\tau=\infty)\pr(\overline{\tau}_1=\infty)}{2\pi\sqrt{\sigma^2\widehat{\sigma}^2k(n-k)}}
\exp\left\lbrace-\frac{(x-k\widehat{a})^2}{2\widehat{\sigma}^2k}-\frac{(x-a(n-k))^2}{2\sigma^2(n-k)}\right\rbrace
\Sigma_N\\
&\hspace{3cm}+o\left(\frac{1}{\sqrt{k(n-k)}}\right),
\end{align}
$$
\Sigma_N:=\sum_{y=1}^N\pr(\overline{\tau}_y=\infty)\pr(X_1\leq-y).
$$
Clearly,
$$
\Sigma_N\to \sum_{y=1}^\infty\pr(\overline{\tau}_y=\infty)\pr(X_1\leq-y)\quad\text{as }N\to\infty.
$$
Plugging now \eqref{T8.5} and \eqref{T8.4} into \eqref{T8.0} and letting $N\rightarrow\infty$, we conclude that
\begin{align*}
&e^{\lambda x}\pr(M_\tau=x,\theta_\tau=k,\tau=n+1)=\\
&\frac{Q}{2\pi\sqrt{k(n-k)}}\exp\left\lbrace-\frac{(x-k\widehat{a})^2}{2\widehat{\sigma}^2k}-\frac{(x-a(n-k))^2}{2\sigma^2(n-k)}\right\rbrace+o\left(\frac{1}{\sqrt{k(n-k)}}\right),
\end{align*}
where
\begin{equation}\label{T8.6}
Q=\frac{\widehat{\pr}^2(\tau=\infty)\pr(\bar{\tau}_1=\infty)}{\widehat{\sigma}\sigma}
\sum_{y=1}^\infty\pr(\bar{\tau_y}=\infty)\pr(X_1\leq-y).
\end{equation}

\subsection{Proof of Corollary~\ref{cor:tau-theta}}
We first prove \eqref{theta-local}. Using \eqref{T8.1a}, we obtain
\begin{align*}
\mathbf{P}(\theta_\tau=k, M_\tau=x)
&=\mathbf{P}(S_k=x,M_{k-1}<x,\tau>k)\mathbf{P}(\tau_+>\tau_x)\\
&=e^{-\lambda x}\widehat{\mathbf{P}}(S_k=x,M_{k-1}<x,\tau>k)\mathbf{P}(\tau_+>\tau_x).
\end{align*}
It is obvious that $\mathbf{P}(\tau_+>\tau_x)\sim\mathbf{P}(\tau_+=\infty)$ as $x\to\infty$.
From this relation and from Theorem~\ref{snmax} with $y=z=0$ we have 
\begin{align*}
e^{\lambda x}\mathbf{P}(\theta_\tau=k, M_\tau=x)
&=\frac{c}{\sqrt{k}}e^{-(x-\widehat{a}k)^2/2\widehat\sigma^2k}
+o\left(\frac{1}{\sqrt{k+x}}\right)\\
&=\frac{c\sqrt{\widehat{a}}}{\sqrt{x}}e^{-\widehat{a}^3(k-x/\widehat{a})^2/2\widehat\sigma^2x}
+o\left(\frac{1}{\sqrt{k+x}}\right).
\end{align*}
Combining this expression with \eqref{Cramer.local}, we conclude that, uniformly in $k$,
$$
\sqrt{x}\mathbf{P}(\theta_\tau=k| M_\tau=x)
=\frac{c\sqrt{\widehat{a}}}{c_0}e^{-\widehat{a}^3(k-x/\widehat{a})^2/2\widehat\sigma^2x}+o(1).
$$
Summing over all $k$ satisfying  $|k-x/\widehat{a}|\le A\sqrt{x}$
and letting $A\to\infty$, we obtain
$$
\lim_{A\to\infty}\lim_{x\to\infty}
\mathbf{P}(|\theta_\tau-x/\widehat{a}|\le A\sqrt{x}|M_\tau=x)
=\frac{c\sqrt{2\pi\widehat{\sigma}^2}}{c_0\widehat{a}}.
$$
It remains to note that \eqref{loc.cond} implies that the left hand side in the previous relation equals $1$. Therefore,
$c=c_0\widehat{a}/\sqrt{2\pi\widehat\sigma^2}$ and
\eqref{theta-local} is proven.

By the total probability formula,
\begin{align*}
\mathbf{P}(\tau-\theta_\tau=j, M_\tau=x)
&=\sum_{k=1}^\infty\mathbf{P}(\theta_\tau=k,\tau=k+j,M_\tau=x)\\
&=\sum_{k=1}^\infty\mathbf{P}(S_k=x,M_{k-1}<x,\tau>k)
\mathbf{P}(\tau_+>\tau_x=j)
\end{align*}
and
\begin{align*}
\mathbf{P}(M_\tau=x)=
\sum_{k=1}^\infty\mathbf{P}(S_k=x,M_{k-1}<x,\tau>k)
\mathbf{P}(\tau_+>\tau_x).
\end{align*}
Therefore, uniformly in $j$,
\begin{align}
\label{T8.6b}
\nonumber
 \mathbf{P}(\tau-\theta_\tau=j| M_\tau=x)
 &=\frac{\mathbf{P}(\tau_+>\tau_x=j)}{\mathbf{P}(\tau_+>\tau_x)}\\
 &\sim
 \frac{\mathbf{P}(\tau_+>\tau_x=j)}{\mathbf{P}(\tau_+=\infty)},
 \quad x\to\infty.
\end{align}
Furthermore,
\begin{align*}
\mathbf{P}(\tau_+>\tau_x=j)
=\mathbf{P}(\overline{M}_j\ge x,\overline{M}_{j-1}<x,\overline{\tau}_1>j).
\end{align*}
If $j\le x/2a$ then, by the Kolmogorov inequality,
\begin{equation}
\label{T8.6a}
\mathbf{P}(\tau_+>\tau_x=j)\le \mathbf{P}(\overline{M}_j\ge x)
=O\left(\frac{1}{x}\right).
\end{equation}
Therefore, it remains to prove \eqref{theta-tau-local} for 
$j> x/2a$. For such values of $j$ we shall use the following
representation:
\begin{align}
\label{T8.7}
\nonumber
\mathbf{P}(\tau_+>\tau_x=j)
&=\mathbf{P}(\overline{M}_j\ge x,\overline{M}_{j-1}<x,\overline{\tau}_1>j)\\
&=\sum_{y=1}^{x}\mathbf{P}(\overline{M}_{j-1}<x,
\overline{S}_{j-1}=x-y,\overline{\tau}_1>j-1)\mathbf{P}(X_1\le-y).
\end{align}

Fix some $N\ge1$. Using \eqref{concentration}, we obtain  
\begin{align}
\label{T8.8}
\nonumber
&\sum_{y=N}^{x}\mathbf{P}(\overline{M}_{j-1}<x,
\overline{S}_{j-1}=x-y,\overline{\tau}_1>j-1)
\mathbf{P}(X_1\le-y)\\
\nonumber
&\hspace{2cm}\le \sum_{y=N}^{x}\mathbf{P}(\overline{S}_{j-1}=x-y)\mathbf{P}(X_1\le-y)\\
&\hspace{2cm}\le\frac{c_1}{\sqrt{j-1}}\sum_{y=N}^{x}
\mathbf{P}(X_1\le-y)
\le\frac{c_1}{\sqrt{x}}\sum_{y=N}^{\infty}
\mathbf{P}(X_1\le-y).
\end{align}
Applying Theorem~\ref{snmax} to the random walk
$\{\overline{S}_n\}$, we have
\begin{align}
\label{T8.9}
\nonumber
&\sum_{y=1}^{N-1}\mathbf{P}(\overline{M}_{j-1}<x,
\overline{S}_{j-1}=x-y,\overline{\tau}_1>j-1)
\mathbf{P}(X_1\le-y)\\
&\hspace{0.5cm}=\frac{\mathbf{P}(\overline{\tau}_1=\infty)}{\sqrt{2\pi\sigma^2 j}}e^{-(x-aj)^2/2\sigma^2 j}
\sum_{y=1}^{N-1}\mathbf{P}(\overline{\tau}_y=\infty)
\mathbf{P}(X_1\le-y)+o\left(\frac{1}{\sqrt{x+j}}\right).
\end{align}
Combining \eqref{T8.7}---\eqref{T8.9} and letting $N\to\infty$,
we arrive at the relation
\begin{align}
\label{T8.10}
\nonumber
\mathbf{P}(\tau_+>\tau_x=j)
&=\frac{c}{\sqrt{2\pi\sigma^2 j}}e^{-(x-aj)^2/2\sigma^2 j}
+o\left(\frac{1}{\sqrt{x+j}}\right)\\
&=\frac{c\sqrt{a}}{\sqrt{2\pi\sigma^2 x}}
e^{-a^3(j-x/a)^2/2\sigma^2 x}
+o\left(\frac{1}{\sqrt{x}}\right),
\quad j\ge x/2a.
\end{align}
Plugging \eqref{T8.10} into \eqref{T8.6b} and taking into account \eqref{T8.6a} we conclude that, uniformly in $j$,
$$
\mathbf{P}(\tau-\theta_\tau=j|M_\tau=x)
=\frac{c'}{\sqrt{2\pi\sigma^2 x}}
e^{-a^3(j-x/a)^2/2\sigma^2 x}
+o\left(\frac{1}{\sqrt{x}}\right).
$$
Thus, it remains to show that $c'=a^{3/2}$.
It suffices to repeat the argument from the proof of \eqref{theta-local} and to notice that
\begin{align*}
&\lim_{A\to\infty}\lim_{x\to\infty}
\mathbf{P}(|\tau-\theta_\tau-x/a|>A\sqrt{x}|M_\tau=x)\\
&\hspace{2cm}\le\lim_{A\to\infty}\lim_{x\to\infty}
\frac{\mathbf{P}(|\tau_x-x/a|>A\sqrt{x})}
{\mathbf{P}(\tau_+=\infty)}=0,
\end{align*}
which follows from \eqref{T8.6b} and from the Kolmogorov inequality. This completes the proof of \eqref{theta-tau-local}.


\begin{thebibliography}{99}
\bibitem{Asmussen82} Asmussen, S.
\newblock Conditioned limit theorems relating a random walk to its associate, with applications to risk reserve processes and the $GI/G/1$ queue.
\newblock  {\em Adv. Appl. Prob.}, {\bf 14}:143-170, 1982.
                 
\bibitem{Doney89} Doney, R.A.
\newblock On the asymptotic behaviour of first passage times for transient random walk.
\newblock {\em  Probab. Theory Related Fields}, {\bf 81}:239-246, 1989.


\bibitem{Igl72} Iglehart, D.L.
\newblock extreme values in the  $GI/G/1$ queue.
\newblock {\em Ann. Math. Statist.}, {\bf 43}: 627-635, 1972.

\bibitem{DonIgl} Iglehart, D.L.
\newblock Functional Central Limit Theorems for Randem Walks Conditioned to Stay Positive.
\newblock {\em Ann. Probab.}, {\bf 2}: 608-619, 1974.

\end{thebibliography}
\end{document}